\documentclass[12pt,reqno]{article}

\usepackage[usenames]{color}
\usepackage{amssymb}
\usepackage{graphicx}
\usepackage{amscd}

\usepackage[colorlinks=true,
linkcolor=webgreen,
filecolor=webbrown,
citecolor=webgreen]{hyperref}

\definecolor{webgreen}{rgb}{0,.5,0}
\definecolor{webbrown}{rgb}{.6,0,0}

\usepackage{color}
\usepackage{fullpage}
\usepackage{float}

\usepackage{graphics,amsmath,amssymb}
\usepackage{amsthm}
\usepackage{amsfonts}
\usepackage{latexsym}
\usepackage{epsf}

\newcommand{\seqnum}[1]{\href{http://www.research.att.com/cgi-bin/access.cgi/as/~njas/sequences/eisA.cgi?Anum=#1}{\underline{#1}}}

\newtheorem{theorem}{Theorem}

\newtheorem{lemma}[theorem]{Lemma}
\newtheorem{corollary}{Corollary}

\newtheorem{conjecture}{Conjecture}

\long\def\symbolfootnote[#1]#2{\begingroup
\def\thefootnote{\fnsymbol{footnote}}\footnote[#1]{#2}\endgroup}

\newcommand{\sg}{\sigma}

\def\W{\mathcal{W}}
\def\S{\mathcal{S}}
\def\A{\mathcal{A}}

\def\F{\mathcal{F}}
\def\P{\mathbb{P}}

\title{Generating functions for Wilf equivalence under generalized factor order}

\author{
Thomas Langley  \\
\small Department of Mathematics \\[-0.8ex]
\small Rose-Hulman Institute of Technology \\[-0.8ex]
\small Terre Haute, IN 47803 USA\\[-0.8ex]
\small \texttt{langley@rose-hulman.edu}
\and
Jeffrey Liese \\
\small Department of Mathematics\\[-0.8ex]
\small California Polytechnic State University\\[-0.8ex]
\small San Luis Obispo, CA 93407-0403. USA\\[-0.8ex]
\small \texttt{jliese@calpoly.edu} \and
Jeffrey Remmel \\
\small Department of Mathematics\\[-0.8ex]
\small University of California, San Diego\\[-0.8ex]
\small La Jolla, CA 92093-0112. USA\\[-0.8ex]
\small \texttt{remmel@math.ucsd.edu} }

\date{\small 05.24.10\\
\small \mbox{}\\
\small MR Subject Classifications: 05A15, 68R15, 06A07
}

\begin{document}
\maketitle

\begin{abstract}

Kitaev, Liese, Remmel, and Sagan recently defined generalized factor order on words 
comprised of letters from a partially ordered set $(P, \leq_P)$  by setting $u \leq_P w$
if there is a subword $v$ of $w$ of the same length as $u$ 
such that the $i$-th character of $v$ is
greater than or equal to the $i$-th character of $u$ for all $i$.   This subword $v$ is called an 
embedding of $u$ into $w$.
For the case where $P$ is the positive integers with the usual ordering, they defined the 
weight of a word $w = w_1\ldots w_n$ to be $\text{wt}(w) = x^{\sum_{i=1}^n w_i} t^{n}$, 
and the corresponding weight generating function $F(u;t,x) = \sum_{w \geq_P u}  \text{wt}(w)$. 
They then defined  two words $u$ and $v$ to be Wilf equivalent, denoted $u \backsim v$, 
 if and only if $F(u;t,x) = F(v;t,x)$.  They also defined the related
generating function $S(u;t,x) = \sum_{w \in \mathcal{S}(u)} \text{wt}(w)$ where $\mathcal{S}(u)$
is the set of all words $w$ such that the only embedding of $u$ into $w$ is a suffix of $w$, and 
showed that $u \backsim v$ if and only if $S(u;t,x) = S(v;t,x)$.  
We continue this study by giving an explicit formula for $S(u;t,x)$ if $u$ factors into a weakly increasing 
word followed by a weakly decreasing word.  We use this formula as an aid to classify Wilf equivalence 
for all words of length 3.  We also show that coefficients of related generating functions are 
well-known sequences in several special cases.  
Finally, we discuss a conjecture that if $u \backsim v$ then $u$ and $v$ must be rearrangements, and the 
stronger conjecture that there also must be a
 weight-preserving bijection $f: \mathcal{S}(u) \rightarrow \mathcal{S}(v)$ 
such that $f(u)$ is a rearrangement of $u$ for all $u$.

\end{abstract}

\section{Introduction and definitions} \label{sec:intro}

Kitaev, Liese, Remmel, and Sagan \cite{KLRS} recently introduced the 
generalized factor order on words comprised of letters from a partially ordered set (poset). That is, 
let  $\mathcal{P} =(P, \leq_P)$ be a poset and let $P^*$ be the Kleene closure of $P$ so that 
$$
P^* = \{w = w_1 w_2 \ldots w_n \: | \: n \geq 0 \text{ and } w_i \in P \text{ for all } i \}.
$$
For $w \in P^*$, let $|w|$ denote 
the number of characters in $w$.  Then for any $u, w \in P^*$, $u$ is less than or equal to $w$ in the 
\textit{generalized factor order} relative to $\mathcal{P}$, written 
$u \leq_{\mathcal{P}} w$, if
there is a string $v$ of $|u|$ consecutive characters in $w$ such that 
the $i$-th character of $v$ is greater than or equal to the $i$-th
 character of 
$u$ under $\leq_\mathcal{P}$ for each $i$, $1 \leq i \leq |u|$.  If $u \leq_{\mathcal{P}} w$,  we will also say that $w$ \textit{embeds} $u$, and that  $v$ is an \textit{embedding of} $u$ \textit{into}  $w$.
We will primarily be interested in the poset $\mathcal{P}_1  = (\P, \leq)$, where $\P$ is the set of positive integers and 
$\leq$ is the usual total order on $\P$.  In this case, for example,  $u = 321 \leq_{\mathcal{P}_1} w = 142322$,
and 423 and 322 are embeddings of $u$ into $w$.  Kitaev, Liese, Remmel, and Sagan \cite{KLRS} 
noted that generalized factor order is related to 
generalized subword order, in which the characters of  $v$ are not required to be adjacent \cite{SV}.

Kitaev, Liese, Remmel, and Sagan \cite{KLRS}  defined Wilf equivalence under 
the generalized factor order on the positive integers in the following way.  For 
$w = w_1 \ldots w_n \in \P^*$,  let 
$\Sigma(w) = \sum_{i=1}^n w_i$ and define the \emph{weight} of $w$ to be
$\text{wt}(w) = t^{n} x^{\Sigma(w)}$.  Then define 
$$
\F(u) = \{w \in \P^* \: | \: u \leq_{\mathcal{P}_1} w \},
$$
and the related generating function 
$$
F(u;t,x) =
\sum_{w \in \F(u)} \text{wt}(w).
$$
Two words $u, v \in \P^*$ are then said to 
be \emph{Wilf equivalent}, denoted $u \backsim v$, 
 if and only if
$F(u;t,x) = F(v;t,x)$.  Kitaev, Liese, Remmel, and Sagan \cite{KLRS} noted that this idea, while inspired by 
the notion of Wilf equivalence used in the theory of pattern avoidance,  is different, since the 
partial order in question is not that of pattern containment.  More information about Wilf 
equivalence in the pattern avoidance context is contained in the survey article by Wilf \cite{Wilf}.

In proving results about Wilf equivalence, it is often convenient to study the sets
\begin{eqnarray*}
\S(u) & = & \{w \in \P^* \: | \: u \leq_{\mathcal{P}_1} w \text{ and the last } |u|  \text{ characters of } w \text{ is the only } \\
&& \hspace{.62in} \text{ embedding of } u \text{ into } w \},  \\
\W(u) & = & \{w \in \P^* \: | \: u \leq_{\mathcal{P}_1} w \text{ and } |w|=|u| \},\text{ and}\\
\A(u) & = & \{w \in \P^* \: | \: u \not \leq_{\mathcal{P}_1} w \}
\end{eqnarray*}
and the corresponding weight generating functions
\begin{eqnarray*}\label{basicgf}
S(u;t,x) &=&  \sum_{w \in \S(u)} \text{wt}(w), \\
W(u;t,x) &=&  \sum_{w \in \W(u)} \text{wt}(w),  \text{ and } \\
A(u;t,x) &=&  \sum_{w \in \A(u)} \text{wt}(w). 
\end{eqnarray*}

Kitaev, Liese, Remmel, and Sagan \cite{KLRS} proved that $F(u;t,x)$, $S(u;t,x)$, 
and $A(u;t,x)$ are rational.  They constructed a non-deterministic finite 
automaton for each $u \in \P^*$ that recognizes $\S(u)$, 
implying that $S(u;t,x)$ is rational.  That the others are rational follows 
from the fact that the weight generating function for all words in $\P^*$ is
\begin{eqnarray*}
\sum_{w \in \P^*} \text{wt}(w) & = & \frac{1}{1-\sum_{n \geq 1} tx^n} \\
& = & \frac{1}{1 - tx/(1-x) } \\
& = & \frac{1-x}{1-x-tx}, 
\end{eqnarray*}
and therefore 
\begin{equation}\label{eq:FS}
F(u;t,x) = S(u;t,x) \frac{1-x}{1-x - tx} 
\end{equation}
and 
$$
F(u;t,x) = \frac{1-x}{1-x-tx} - A(u;t,x).
$$
We also note that $W(u;t,x)$ is 
rational since
$$
W(u;t,x) = \frac{t^{|u|}x^{\Sigma(u)}}{(1-x)^{|u|}}.
$$

From (\ref{eq:FS}), we have that  $F(u;t,x) = F(v;t,x)$ if and only if $S(u;t,x) = S(v;t,x)$, 
and therefore $u \backsim v$ if and only if $S(u;t,x) = S(v;t,x)$.  Much of our 
work will be centered around computing explicit formulas for $S(u;t,x)$ for certain 
words $u$.  In particular, 
Kitaev, Liese, Remmel and Sagan \cite{KLRS} gave two examples 
of classes of words $u$ such that $S(u;t,x)$ has a
simple form.  That is, 
they proved that if $u = 1~2~3 \ldots n-1~n$ or $u =1^kb^\ell$ for 
some $k \geq 0$, $\ell \geq 1$, and $b \geq 2$, then 
$S(u;t,x) = \frac{x^s t^r}{P(u;t,x)}$ for some polynomial $P(u;t,x)$, and 
produced an explicit expression for $P(u;t,x)$ in each case. 
We shall show that there is a much richer class of  
of words $u$ such that $S(u;t,x)$ has this same form. Specifically, 
for any word $u$, let 
$u_{inc}$ be the longest weakly increasing prefix of $u$.  If 
$u = u_{inc}v$ and $v$ is weakly decreasing, then we shall 
say that $u$ has an \textit{increasing/decreasing factorization} and 
denote $v$ as $u_{dec}$.  Thus if $u = u_1 u_2 \ldots u_n$ 
has an increasing/decreasing factorization, then either 
$u_1 \leq \cdots \leq u_n$, in which case $u_{inc}= u$ and 
$u_{dec}$ is the empty string $\varepsilon$, or 
there is a  $k <n$ such that $u_1 \leq \cdots \leq u_k > u_{k+1} \geq \cdots \geq 
u_n$, in which case $u_{inc} = u_1 \ldots u_k$ and $u_{dec} = 
u_{k+1} \ldots u_n$.  For the theorem that follows, we define 
$$
D^{(i)}(u) = \{n-i+j: 1 \leq j \leq i \text{ and } u_j > u_{n-i+j}\}
$$
and $d_i(u) = \sum_{n-i+j \in D^{(i)}(u)} (u_j - u_{n-i+j})$.
For example, if $u = 1~2~3~4~4~3~1~1$ and $i=5$, then 
by considering the diagram 
$$
\begin{array}{cccccccc}
1 & 2 & 3 & 4 & 4 & 3 & 1 & 1 \\
   &    &    & 1 & 2 & 3 & \underline{4} & \underline{4}    \end{array}
$$ 
we see that 
$D^{(5)}(u) = \{7,8\}$ and $d_5(u) = (4-1) + (4-1) = 6$. 
One of our main results is the following theorem.


\begin{theorem}\label{thm:S}
Let $u=u_1 u_2 \ldots u_n \in \P^*$ have an increasing/decreasing factorization. For $1 \leq i \leq n-1$, let 
$s_i = u_{i+1} u_{i+2} \ldots u_n$ and $d_i = d_i(u)$.  Also let  $s_{n} = \varepsilon$ and 
$d_n = 0$.  
Then 
$$
S(u;t,x) = \frac{t^{n} x^{\Sigma(u)}}{t^{n} x^{\Sigma{(u)}}+(1-x-tx)\sum_{i=1}^{n} t^{n-i} x^{d_i+\Sigma(s_i)}(1-x)^{i-1}}.
$$
\end{theorem}


Since the words $u = 1~2~3 \ldots n-1~n$ or $u =1^kb^\ell$ for 
some $k \geq 0$, $\ell \geq 1$, and $b \geq 2$ clearly have increasing/decreasing factorizations, 
Theorem \ref{thm:S} covers both of the cases 
proved by Kitaev, Liese, Remmel and Sagan \cite{KLRS}. 

Theorem~\ref{thm:S} will lead us to the other main results in our work.  First,
we will use Theorem \ref{thm:S}, as well as a slight modification in a special case,  to 
 completely 
classify the Wilf equivalence classes of $\mathcal{P}_1$ for all 
words of length 3.  We will also compute $S(u;t,x)$, along with $F(u;t,x)$ and $A(u;t,x)$,
for some simple words and show that the coefficients in these generating functions
are often well-known sequences.  
Next, 
Theorem \ref{thm:S} will allow us  
to show that if $u$ and $v$ are words with
increasing/decreasing factorizations, then $u \backsim v$ if and only 
if $v$ is a rearrangement of the letters of $u$.  This shows
that words with increasing/decreasing factorizations satisfy 
the following conjecture of Kitaev, Liese, Remmel, and Sagan \cite{KLRS}.

\begin{conjecture}[Kitaev, Liese, Remmel, Sagan]\label{con:weak}
If $u \backsim v$, then $v$ is a rearrangement of $u$. 
 \end{conjecture}
 
We shall call this conjecture the {\em weak rearrangement conjecture}. 
In fact, we conjecture something much stronger is true. 
 
\begin{conjecture} \label{con:strong}
If $u \backsim v$, then there is a weight preserving bijection 
$f:\P^* \rightarrow \P^*$ such that for all $w \in \P^*$, 
$f(w)$ is a rearrangement of $w$ and 
$w \in \F(u) \iff f(w) \in \F(v)$. 
\end{conjecture}

 We will call such a bijection $f$ 
a \emph{rearrangement map} that \emph{witnesses} $u \backsim v$ and 
refer to this conjecture as the {\em strong rearrangement conjecture}. All the Wilf equivalences proved by Kitaev, Liese, Remmel, and Sagan in  Section 4 of \cite{KLRS}
were proved by a constructing a rearrangement map that witnessed the given 
Wilf equivalence.  

We investigate the rearrangement conjectures by considering the class of finite posets  $\mathcal{P}_{[m]} = ([m],\leq)$, where $[m] = \{1, \ldots, m\}$ and 
$\leq$  is the usual total order on $\mathbb{P}$.
For any word $w \in [m]^*$ and $i \in [m]$, let $c_i(w)$ equal the number of occurrences of $i$ in $w$.  
 Then we introduce variables $x_1, x_2, \ldots, x_m$, and define the weight of $w$, $W_{[m]}(w)$, to be $W_{[m]}(w) = \prod_{i=1}^m x_i^{c_i(w)}$.  To define Wilf equivalence
 in this context, we set
\begin{eqnarray*}
F(u; x_1, \ldots , x_m) & = & \sum_{w \in \F(u) \cap [m]^*} W_{[m]}(w), 
\end{eqnarray*}
and define $u, v \in [m]^*$ to be Wilf equivalent with respect to the poset $\mathcal{P}_{[m]}$, denoted $u \backsim_{[m]} v$, if and only $F(u; x_1, \ldots , x_m) = F(v; x_1, \ldots , x_m)$.  
We will also have use for the related generating functions
\begin{eqnarray*}
W(u; x_1, \ldots , x_m) & = & \sum_{w \in \W(u) \cap [m]^*} W_{[m]}(w), \\
S(u; x_1, \ldots , x_m) & = & \sum_{w \in \S(u)\cap [m]^*} W_{[m]}(w), \ \mbox{and}  \\
A(u; x_1, \ldots , x_m) & = & \sum_{w \in \A(u)\cap [m]^*} W_{[m]}(w).
\end{eqnarray*}
Note that we have dropped the $t$ dependence in these generating functions 
since length is recorded by the
number of variables in a monomial.  We now have
$$
\sum_{w \in [m]^*} W_{[m]}(w) = \frac{1}{1 - \sum_{i=1}^m x_i}.
$$
Thus since $\F(u)\cap [m]^* = (\S(u)\cap [m]^*)[m]^*$ and 
$A(u) \cap [m]^* = [m]^* - (\F(u) \cap [m]^*)$, we have that 
\begin{eqnarray*}
F(u; x_1, \ldots , x_m) &=& S(u; x_1, \ldots , x_m)\frac{1}{1 - \sum_{i=1}^m x_i} \ \mbox{and} \\
A(u; x_1, \ldots , x_m) &=& \frac{1}{1 - \sum_{i=1}^m x_i} - F(u; x_1, \ldots , x_m)
\end{eqnarray*} 
so that if any one of $F(u; x_1, \ldots , x_m)$, 
$A(u; x_1, \ldots , x_m)$, or $S(u; x_1, \ldots , x_m)$ is rational, 
then so are the other two. It follows from Theorem 8.2 of 
\cite{KLRS} that $S(u; x_1, \ldots , x_m)$ is rational for 
all $m \geq 1$, so  $F(u; x_1, \ldots , x_m)$ and 
$A(u; x_1, \ldots , x_m)$ 
are also rational for all $m \geq 1$. Also note that if 
$u =u_1 \ldots u_n$, then 
$$
W(u;x_1, \ldots ,x_m) = \prod_{r=1}^n \sum_{s=u_r}^m x_s,
$$
so $W(u;x_1, \ldots ,x_m)$ is rational.   We will show 
that if $u \backsim_{[m]} v $ for some $m$, then there is a 
rearrangement map that witnesses the equivalence $u \backsim v$.  This gives us 
a way to test the strong rearrangement conjecture for any 
particular pair of words $u,v \in \P^*$.  We will also give an analogue
of Theorem~\ref{thm:S} for these posets.

The outline of this paper is as follows. In Section~\ref{sec:S}, 
we prove Theorem \ref{thm:S}  and show that the weak  rearrangement conjecture
holds for words with increasing/decreasing factorizations. In Section~\ref{sec:sequences},
we
compute  $F(u;t,x)$, $S(u;t,x)$ and $A(u;t,x)$ for some simple words. In particular, 
we show that the sequences of coefficients that arise in the expansions around $x = 0$ of 
$F(k;1,x)$, $S(k;1,x)$, $A(k;1,x)$, $S(1k1;1,x)$ and $A(1k1;1,x)$ as $k$ varies 
have appeared in 
the On-line Encyclopedia of Integer Sequences (OEIS).
We
follow this with the classification of the Wilf equivalence classes 
of words of length 3 in Section~\ref{sec:3}. The results 
of Sections \ref{sec:S} and \ref{sec:3} allow us to  compute 
$S(\sigma;t,x)$ and $A(\sigma;t,x)$ for all permutations 
in the symmetric group $S_3$ as there are only two Wilf equivalences 
classes for such permutations. In these cases,   
the coefficients that arise in 
the expansions of $S(\sigma;1,x)$ and $A(\sigma;1,x)$ 
around $x=0$ do not correspond to any sequences that have appeared 
in the OEIS.  We discuss the strong rearrangement conjecture in 
Section~\ref{sec:rearrangement}, 
as well as the analogue of Theorem~\ref{thm:S} for the posets $\mathcal{P}_{[m]}$.
We conclude with a few remarks about further work in Section~\ref{sec:conclusion}.


\section{Words such that $S(u;t,x) = \frac{x^st^r}{P(u;t,x)}$ where $P(u;t,x)$ 
is a polynomial.} \label{sec:S}
In this section we prove Theorem \ref{thm:S}, and show that Conjecture~\ref{con:weak} holds for words with an increasing/decreasing factorization. \\

\noindent \textit{Proof of Theorem~\ref{thm:S}.}
Let $u=u_1 u_2 \ldots u_n \in \P^*$ have an increasing/decreasing factorization. If $w = w_1 \ldots w_m \in \mathcal{S}(u)$, then $w_1 \ldots w_{m-n} \in \A(u)$ and $u \leq 
w_{m-n+1} \ldots w_{m}$. However if 
$v \in \A(u)$ and $z = z_1 \ldots z_n$ is such that 
$u \leq z$, then it may not be the case that 
$w = vz \in \S(u)$ because there might be another 
embedding of $u$ in the last $2n-1$
 letters of $w$, starting in $v$ and ending in $z$. 
 Of course, there can be 
no embedding of $u$ which starts to the left of the last $2n-1$ letters of 
$w$ since $v \in \A(u)$. 
For each $1 \leq i \leq n-1$, 
we define $\S^{(i)}(u)$ to be set of all words $w = w_1 \ldots w_m$ such 
that
\begin{description}
\item[(i)] $u \leq w_{m-n+1} \ldots w_m$ (so that $u$ embeds into the suffix 
of length $n$ of $w$) 
and 
\item[(ii)] the left-most embedding of $u$ into $w$ starts at position 
$m-2n+i$.
\end{description}
We then let 
$$
S^{(i)}(u;t,x) = \sum_{w \in \S^{(i)}(u)} \text{wt}(w) =  \sum_{w \in \S^{(i)}(u)} x^{\Sigma(w)} t^{|w|}.
$$
Thus  
\begin{equation}\label{eq:S2}
 \S(u) = \A(u)\mathcal{W}(u)- \bigcup_{i=1}^{n-1} \S^{(i)}(u).
\end{equation}

Now,
\begin{eqnarray}\label{eq:S3}
\sum_{w \in \A(u)\mathcal{W}(u)} x^{\Sigma(w)} t^{|w|} &=& 
A(u;t,x)\frac{t^nx^{\Sigma(u)}}{(1-x)^n} \nonumber \\
&=& \frac{(1-x)}{(1-x-tx)}(1-S(u;t,x))\frac{t^nx^{\Sigma(u)}}{(1-x)^n}.
\end{eqnarray}
We claim that we have the following lemma. 
\begin{lemma}\label{lem:S}
 Let $u=u_1 u_2 \ldots u_n \in \P^*$ have an increasing/decreasing factorization.   Then 
for $1 \leq i \leq n-1$,  
$$
S^{(i)}(u;t,x) = S(u;t,x)  t^{n-i}x^{d_i+\Sigma(s_i)} \left(\frac{1}{1-x}\right)^{n-i}.
$$
\end{lemma}

Given Lemma \ref{lem:S}, it is easy to complete the 
proof of Theorem \ref{thm:S}. That is, our definitions ensure that
 $\S^{(1)}(u), \S^{(2)}(u), \ldots,  
\S^{(n-1)}(u)$ are pairwise disjoint, so that 
\begin{eqnarray*}
\sum_{w \in \bigcup_{i=1}^{n-1} \S^{(i)}(u)} x^{\Sigma(w)} t^{|w|} &=& 
\sum_{i=1}^{n-1} S^{(i)}(u;t,x) \\
&=& S(u;t,x) \sum_{i=1}^{n-1}  t^{n-i}x^{d_i+\Sigma(s_i)} \left(\frac{1}{1-x}\right)^{n-i}.
\end{eqnarray*}
Thus it follows from (\ref{eq:S2}) and (\ref{eq:S3}) that 
\begin{eqnarray*}
S(u;t,x) & = & \frac{(1-x)}{(1-x-tx)}(1-S(u;t,x))\frac{t^nx^{\Sigma(u)}}{(1-x)^n}\\
&&- S(u;t,x) \sum_{i=1}^{n-1} \frac{x^{d_i+\Sigma(s)_i}t^{n-i}}{(1-x)^{n-i}}.
\end{eqnarray*}
Solving for $S(u;t,x)$ will yield the result in the theorem.

Thus we need only prove Lemma \ref{lem:S}.  To this end, fix $i$, 
$1 \leq i \leq n-1$, and suppose that $w =w_1 \ldots w_m \in \S^{(i)}(u)$. 
If $\bar{w}= w_1 \ldots w_{m-n+i}$, then our definitions ensure 
that 
\begin{enumerate}
\item $\bar{w} \in \S(u)$, 
\item $u_1 \ldots u_i \leq w_{m-n+1} \ldots w_{m-n+i}$ and 
\item $s_i = u_{i+1} \ldots u_n \leq w_{m-n+i+1} \ldots w_m$. 
\end{enumerate}

Now, the generating function of all words $v$ of length $n-i$ such 
that $s_i \leq v$ is $\frac{x^{\Sigma (s_i)} t^{n-i}}{(1-x)^{n-i}}$.
So let $\bar{\S}^{(i)}(u)$ denote the set of all words 
$\bar{w}$ that satisfy conditions 1 and 2,  and let 
$$\bar{S}^{(i)}(u;t,x) = \sum_{\bar{w} \in \bar{\S}^{(i)}(u)} x^{\Sigma(\bar{w})}t^{|\bar{w}|}.$$ 
Then 
$$
S^{(i)}(u;t,x) =\bar{S}^{(i)}(u;t,x)\frac{x^{\Sigma (s_i)} t^{n-i}}{(1-x)^{n-i}}.
$$
Thus we need only show that 
\begin{equation}\label{eq:S6}
\bar{S}^{(i)}(u;t,x) = x^{d_i} S(u;t,x).
\end{equation}
Now suppose that $v = v_1 \ldots v_p \in \bar{\S}^{(i)}(u)$. Then 
let $\tilde{v} = \tilde{v}_1 \ldots \tilde{v_p}$ be the 
word that results from $v$ by decrementing $v_{p-i+j}$ by $u_j - u_{n-i+j}$ if $n-i+j \in  D^{(i)}(u)$ and leaving all other letters the same.  If 
$n-i+j \in D^{(i)}(u)$, then  $v_{p-i+j} \geq u_j$, and hence
$\tilde{v}_{p-i+j} \geq u_{n-i+j}$. Thus it will still be the case that  
$u$ embeds in the final segment of $\tilde{v}$ of length $n$ so that 
$\tilde{v} \in \S(u)$. Thus to complete the proof of (\ref{eq:S6}), 
we need only show that if we start with a word 
$\tilde{v} = \tilde{v}_1 \ldots \tilde{v_p}$ in $\S(u)$ and 
create a new word $v =v_1 \ldots v_p$ by incrementing $\tilde{v}_{p-i+j}$ by $u_j - u_{n-i+j}$ if $n-i+j \in  D^{(i)}(u)$ and leaving all other letters the same, then $v \in \bar{\S}^{(i)}(u)$.  Clearly $v$ satisfies condition (2) above.  
The only question is whether $v$ is still in $\S(u)$.  That is, 
since we have incremented some letters in $\tilde{v}$ to get $v$, we 
might have  created a 
new embedding of $u$ which starts to the left of position 
$p-n+1$.  If so, any such embedding must contain 
at least one position of the form $p-i+j$ where $n-i+j \in D^{(i)}(u)$. However if $u_r$ is the letter in this new embedding of $u$ into $v$ 
which corresponds to position 
$p-i+j$, then $r$  must be strictly greater than $n-i+j$. 
But if $u= u_1 \leq \cdots \leq u_k > u_{k+1} \geq \cdots \geq u_n$, 
then it must always be the case that $D^{(i)}(u) \subseteq \{k+1, \ldots, n\}$. That is, if $j < n-i+j \leq k$, then $u_j \leq u_{n-i+j}$ and 
hence $n-i+j \not \in D^{(i)}(u)$. 
Hence $u_{n-i+j} \geq u_r$. 
But then $\tilde{v}_{p-i+j} \geq u_{n-i+j} \geq u_r$ which would mean 
that there would have been an embedding of $u$ into $\tilde{v}$ which 
started to the left of $p-n+1$.  Since $\tilde{v}$ was assumed to 
be in $\S(u)$, there can be no such embedding and hence 
$v \in \S(u)$. Thus (\ref{eq:S6}) holds and the lemma is proved.  \qed \\

To  illustrate the ideas in the proof, consider  $u = 1\:2\:6\:5\:3\:2$,
so that $u_{inc} = 1\: 2\: 6$ and $u_{dec} = 5\:3\:2$, and let $i=5$. Then elements of $\bar{\S}^{(5)}(u)$ must end in an embedding of $u$ in 
the final six characters and an embedding of $1\:2\:6\:5\:3$ in the final five characters, as shown:
$$
\begin{array}{ccccccccc}
\tilde{v}= & \cdots  & \bullet & \bullet & \bullet & \star & \star & \star \\
& & 1& 2 & 6 & 5 & 3 & 2 \\
& & & 1& 2 & 6 & 5 & 3 \\
\end{array},
$$
where the stars indicate the positions in $\tilde{v}$ that must be increased to form $v$. 
Note that the stars all embed characters of $u_{dec}$, and that  $d_5 = (6-5)+(5-3)+(3-2) =4$.  
If $v$ were to contain a new embedding of $u$ to the left of the first original embedding, that
new embedding must end in the second or third position from the end:
$$
\begin{array}{cccccccccc}
v= & \cdots  &  \bullet & \bullet & \bullet & \bullet & \star & \star & \star \\
& &  & 1& 2 & 6 & 5 & 3 & 2   \\
& &  1& 2 & 6 & 5 & 3  & 2 \\
\end{array}
$$
or 
$$
\begin{array}{cccccccccccc}
v= & \cdots  &  \bullet & \bullet & \bullet & \bullet & \bullet & \star & \star & \star  \\
& &  & & 1& 2 & 6 & 5 & 3 & 2 &  \\
& &  1& 2 & 6 & 5 & 3  & 2 & &  & \\
\end{array}.
$$
But in both cases, the characters below the stars
are decreasing, so such an embedding would have already existed 
in $\tilde{v}$.

It's worth noting here that the condition that $u$ has an increasing/decreasing
factorization is necessary for the technique in the proof of Lemma~\ref{lem:S}
to be valid.  That is, if $u$ does not have an increasing/decreasing factorization,
there is always at least one index $i$ where words counted by 
$\bar{S}^{(i)}(u;t,x)$ can not be formed by simply starting 
with a word $\tilde{v} \in \S(u)$ and creating a
word $v =v_1 \ldots v_p$ by incrementing $\tilde{v}_{p-i+j}$ by $u_j - u_{n-i+j}$ if $n-i+j \in  D^{(i)}(u)$ and leaving all other letters the same. 
For example, consider $u=2112$ with $i=2$.   Then $D^{(2)}(u) = \{3\}$ and 
$d_2(u) =1$. However if we start with $\tilde{v} = 
122112 \in \S(u)$ and increment $\tilde{v}_{5}$ to obtain $v$, 
then $v=122122$ which is not in $\S(u)$ because there is an embedding 
of $u$ which starts at position 2.  
The problem here is that the second 1 in $u$ is followed by a larger character, and
also has a larger character to its left.  A similar situation will always occur for 
at least one $i$ when  $u$  does not have an increasing/decreasing factorization. Experimental evidence suggests the following conjecture.

\begin{conjecture}

For $u \in \P^*$, $S(u; t,x)= \frac{x^s t^r}{P(u;t,x)}$ where
$P(u;t,x)$ is a polynomial if and only if $u$ has an increasing/decreasing factorization. 

\end{conjecture}

It is a consequence of Corollary 4.2 in \cite{KLRS} that if 
$u$ and $v$ have increasing/decreasing factorizations and 
$u$ is a rearrangement of $v$, then $u \backsim v$. 
We shall give a new proof of that fact here, as well as prove the 
converse. That is, if $u$ and $v$ both have increasing/decreasing factorizations and $u \backsim v$, then $u$ and $v$ are rearrangements,  showing that 
the weak rearrangement conjecture holds for words with increasing/decreasing factorizations.

We begin with the following lemma.

\begin{lemma}\label{lem:rearrangements}
Suppose $u = u_1 \ldots u_n$ is a rearrangement of  $v=v_1 \ldots v_n$ and 
that $u$ and $v$ have increasing/decreasing factorizations. 
For each $i$,  $1 \leq i \leq n-1$, let 
$s_i(u) = u_{i+1} \ldots u_n$,  $s_i(v) = v_{i+1} \ldots v_n$, 
$d_i(u) = \sum_{n-i+j \in D^{(i)}(u)} (u_j - u_{n-i+j})$, and 
$d_i(v) = \sum_{n-i+j \in D^{(i)}(v)} (v_j - v_{n-i+j})$. Then 
for all $1 \leq i \leq n-1$,
$$
d_i(u)+\Sigma(s_i(u)) = d_i(v)+\Sigma(s_i(v)).
$$
\end{lemma}

\begin{proof}
First suppose that $u = u_1 \ldots u_n$ where $u_1 \leq \cdots \leq u_n$. 
Then for each $i$, $1 \leq i \leq n-1$, $d_i(u) =0$ and $\Sigma(s_i(u)) = 
\sum_{j=i+1}^n u_j$. So it suffices to show that 
$d_i(v)+\Sigma(s_i(v)) = \Sigma(s_i(u))$ for all $1 \leq i \leq n-1$ whenever 
$v$ has an increasing/decreasing factorization and $v$ is a rearrangement 
of $u$. So fix $i$, $1 \leq i \leq n-1$,  and let $\sg = \sg_1 \ldots \sg_n$ be a permutation of $\{1,\ldots, n\}$ such that 
$$v = 
u_{\sg_1} \leq \cdots \leq u_{\sg_j} > u_{\sg_{j+1}} \geq \cdots \geq u_{\sg_n}.$$ 
Then let 
\begin{eqnarray*}
A^i &=& \{s: s \leq i  \text{ and }  u_s \in \{u_{\sg_1}, \ldots ,u_{\sg_j}\}\}, \\ 
B^i &=& \{s: s > i  \text{ and }  u_s \in \{u_{\sg_1}, \ldots ,u_{\sg_j}\}\}, \\ 
C^i &=& \{s: s > i \text{ and } u_s \in \{u_{\sg_{j+1}}, \ldots ,u_{\sg_n}\}\}, 
\  \mbox{and} \\ 
D^i &=& \{s: s \leq i  \text{ and }  u_s \in \{u_{\sg_{j+1}}, \ldots ,u_{\sg_n}\}\}.
\end{eqnarray*}
For example, suppose $u =1~2~3~3~4~5~5~6~7~7$ and 
$\sg = 2~3~4~9~10~8~7~6~5~1$ so that 
$$
\begin{array}{ccccccccccccc}
 v & = & u_2 & u_3 & u_4 & u_9 & u_{10} & u_8 & u_7 & u_6 & u_5 & u_1\\
 & = & 2 & 3 & 3 & 7 & 7 & 6 & 5 & 5 & 4 & 1 
\end{array}
$$
and $j=5$.   Then for $i=6$, $A^6 =\{2,3,4\}$, $B^6 =\{9,10\}$, $C^6 = \{7,8\}$, and 
$D^6 = \{1,5,6\}$. Let $a_i = |A^i|$, $b_i = |B^i|$, $c_i = |C^i|$, and 
$d_i = |D^i|$. Then our definitions force $a_i+d_i =i$, $b_i+c_i = n-i$, 
$a_i + b_i =j$, and $c_i +d_i =n-j$. For any set 
$D = \{d_1 < \cdots < d_r\} \subseteq  \{1, \ldots ,n\}$, let 
\begin{eqnarray*}
D(u)\!\!\uparrow &=& u_{d_1}u_{d_2} \ldots u_{d_r} \ \mbox{and} \\
D(u)\!\!\downarrow &=& u_{d_r}u_{d_{r-1}} \ldots u_{d_1}.
\end{eqnarray*}
Thus $v = A^i(u)\!\!\uparrow B^i(u)\!\!\uparrow C^i(u)\!\!\downarrow D^i(u)\!\!\downarrow$ and 
$\Sigma(B^i(u)\!\!\uparrow) + \Sigma(C^i(u)\!\!\downarrow) =\Sigma(s_i(u))$.
We then have four cases to consider depending  on whether 
$v_i \in  A^i(u)\!\!\uparrow$, $v_i \in  B^i(u)\!\!\uparrow$, $v_i \in  C^i(u)\!\!\downarrow$, or $v_i \in  D^i(u)\!\!\downarrow$. \\
\ \\
{\bf Case 1.} $v_i \in  A^i(u)\!\!\uparrow$.\\
\ \\
In this case, it must be that $i = a_i$ and $A^i(u)\!\!\uparrow =u_1 \ldots u_i$. 
But then $D^i = \emptyset$ and $s_i(v) = B^i(u)\!\!\uparrow C^i(u)\!\!\downarrow$,
a rearrangement of $s_i(u)$.    Moreover,  it will 
be the case that $v_j \leq v_{n-i+j}$ for $j=1, \ldots, i$ so that 
$d_i(v) = 0$. Thus $d_i(v) + \Sigma(s_i(v)) = \Sigma(s_i(u))$ as 
desired.   As an example, with $u$ as in the previous example, consider
$$
\begin{array}{cccccccc|cc|cc}
 v & = & u_1 & u_2 & u_3 & u_4 & u_{5} & u_6 & u_9 & u_{10} & u_8 & u_7\\
 & = & 1 & 2 & 3 & 3 & 4 & 5 & 7 & 7 & 6 & 5 \end{array}
$$
so that $j=8$, and again let $i = 6$.   Then as indicated by the dividers, $A^6(u)\!\!\uparrow= u_1 \ldots u_6$ so that $a_i = i = 6$,  $B^6(u)\!\!\uparrow = u_9 u_{10}$ and 
$C^6(u)\!\!\downarrow = u_8 u_7$.\\
\ \\
{\bf Case 2.} $v_i \in  B^i(u)\!\!\uparrow$.\\
\ \\
In this case $a_i <i \leq 
a_i+b_i$.   For example, with the same $u$ as above, let 
$$
\begin{array}{cccccc|ccc|c|ccc}
 v & = & u_2 & u_3 & u_4 & u_6 & u_{7} & u_9 & u_{10} & u_{8} & u_5 & u_1\\
 & = & 2 & 3 & 3 & 5 & 5 & 7 & 7 & 6 & 4 & 1 \end{array}
$$
so that $j=7$, and again let $i = 6$.  Then
$A^6(u)\!\!\uparrow= u_2 u_3 u_4 u_6$, $B^6(u)\!\!\uparrow = u_7 u_9 u_{10}$, $C^6(u)\!\!\downarrow = u_8$, 
and $D^6(u)\!\!\downarrow = u_5 u_1$, so $v_6 \in B^6(u)\!\!\uparrow$ and $a_6 = 4 < i \leq 7 = a_6 + b_6$.

Now let $B^i_1(u) = v_{a_i+1} \ldots v_i$ and 
$B^i_2(u) = v_{i+1} \ldots v_{a_i+b_i}$. 
Then $s_i(v) = B^i_2(u)C^i(u)\!\!\downarrow D^i(u)\!\!\downarrow$. 
When we compare the first $i$ letters of $v$ with the last $i$ letters of 
$v$, we see that the letters in $B_1^i(u)$ are compared with 
the letters in $D^i(u)\!\!\downarrow$ since 
$|B^i_1(u)| = i-a_i = |D^i(u)\!\!\downarrow|$. But the letters 
in $D^i(u)\!\!\downarrow$ come from $\{u_s:s\leq i\}$ and the letters 
from $B_1^i(u)$ come from $\{u_s:s> i\}$.  Thus any letter in 
$B_1^i(u)$ is greater than or equal to every letter in $D^i(u)\!\!\downarrow$ 
so that such letters will contribute 
$\Sigma(B^i_1(u)) - \Sigma(D^i(u)\!\!\downarrow)$ to $d_i(v)$. However 
the letters in $A^i(u)\!\!\uparrow$ will be compared to letters 
that lie in either $C^i(u)\!\!\downarrow$, $B^i(u)\!\!\uparrow$, or later 
letters in $A^i(u)\!\!\uparrow$, and hence they will contribute 
0 to $d_i(v)$. Thus 
\begin{eqnarray*}
d_i(v)+\Sigma(s_i(v))&=&  \Sigma(B^i_1(u)) - \Sigma(D^i(u)\!\!\downarrow)  +
\Sigma(B^i_2(u)) + \Sigma(C^i(u)\!\!\downarrow) + \Sigma(D^i(u)\!\!\downarrow)  \\
&=&   \Sigma(B^i(u)\!\!\uparrow) + \Sigma(C^i(u)\!\!\downarrow) =\Sigma(s_i(u)).
\end{eqnarray*}
\ \\
{\bf Case 3.} $v_i \in  C^i(u)\!\!\downarrow$.\\
\ \\
In this case $a_i +b_i <i  \leq a_i +b_i +c_i$.   For example, with the same $u$ 
as above let $$
\begin{array}{cccc|cc|cc|cccc}
 v & = & u_2 & u_3 &  u_9 & u_{10} & u_{8} & u_7 & u_6 & u_5 & u_4 & u_1\\
 & = & 2 & 3 & 7 & 7 & 6 & 5 & 5 & 4 & 3 & 1 \end{array}
$$
so that $j=4$, and again let $i = 6$.  Then
$A^6(u)\!\!\uparrow= u_2 u_3$, $B^6(u)\!\!\uparrow = u_9 u_{10}$, $C^6(u)\!\!\downarrow = u_8 u_7$, 
and $D^6(u)\!\!\downarrow = u_6 u_5 u_4 u_1$, so $v_6 \in C^6(u)\!\!\uparrow$ and $a_6 +b_6= 4 < i \leq 6 = a_6 + b_6+c_6$.

Now let $C^i_1(u) = v_{a_i+b_i+1} \ldots v_i$ and $C^i_2(u) = v_{i+1} \ldots v_{a_i+b_i+c_i}$. 
Then $s_i(v) = C^i_2(u)D^i(u)\!\!\downarrow$. 
When we compare the first $i$ letters of $v$ with the last $i$ letters of 
$v$, we see that the letters in $B^i(u)\!\!\uparrow C^i_1(u)$ are compared with 
the letters in $D^i(u)\!\!\downarrow$ since 
$|B^i(u)\!\!\uparrow| + |C^i_1(u)| = i-a_i = |D^i(u)\!\!\downarrow|$. But the letters 
in $D_i^i(u)\!\!\downarrow$ come from $\{u_s:s\leq i\}$ and the letters 
from $B^i(u)\!\!\uparrow C^i_1(u)$ come from $\{u_s:s> i\}$.  Thus any letter in 
$B^i(u)\!\!\uparrow C^i_1(u)$ is greater than or equal to every letter in $D^i(u)\!\!\downarrow$ 
so that such letters will contribute 
$\Sigma(B^i(u)\!\!\uparrow) + \Sigma(C^i_1(u)) - \Sigma(D^i(u)\!\!\downarrow)$ to $d_i(v)$. However 
the letters in $A^i(u)\!\!\uparrow$ will be compared to letters 
that lie in either $C^i(u)\!\!\downarrow$, $B^i(u)\!\!\uparrow$, or later 
letters in $A^i(u)\!\!\uparrow$, and hence they will contribute 
0 to $d_i(v)$. Thus 
\begin{eqnarray*}
d_i(v)+\Sigma(s_i(v))&=&  \Sigma(B^i(u)\!\!\uparrow) + \Sigma(C^i_1(u)) - \Sigma(D^i(u)\!\!\downarrow)  +
\Sigma(C^i_2(u)) + \Sigma(D^i(u)\!\!\downarrow)  \\
&=&   \Sigma(B^i(u)\!\!\uparrow) + \Sigma(C^i(u)\!\!\downarrow) =\Sigma(s_i(u)).
\end{eqnarray*}
\ \\
{\bf Case 4.} $v_i \in  D^i(u)\!\!\uparrow$.\\
\ \\
In this case $a_i +b_i +c_i <i$.  For example, with the same $u$ 
as above, now let $$
\begin{array}{ccc|cc|cc|ccccc}
 v & = & u_2 & u_9 &  u_{10} & u_{8} & u_{7} & u_6 & u_5 & u_4 & u_3 & u_1\\
 & = & 2 & 7 & 7  & 6 & 5 & 5 & 4 & 3 & 3 & 1 \end{array}
$$
so that $j=3$, and once again let $i = 6$.  Then
$A^6(u)\!\!\uparrow= u_2$, $B^6(u)\!\!\uparrow = u_9 u_{10}$, $C^6(u)\!\!\downarrow = u_8 u_7$, 
and $D^6(u)\!\!\downarrow = u_6 u_5 u_4 u_3 u_1$, so $v_6 \in D^6(u)\!\!\uparrow$ and $a_6 +b_6+c_6 = 5 < i$.

Now let $D^i_1(u) = v_{a_i+b_i+c_i+1} \ldots v_i$ and $D^i_2(u) = v_{i+1} \ldots v_n$. 
Then $s_i(v) = D^i_2(u)$. 
When we compare the first $i$ letters of $v$ with the last $i$ letters of 
$v$, we see that the letters in $B^i(u)\!\!\uparrow C^i(u)\!\!\downarrow D_1^i(u)$ are compared with 
the letters in $D^i(u)\!\!\downarrow$ since 
$|B^i(u)\!\!\uparrow| + |C^i(u)\!\!\downarrow|+ |D^i_1(u)| = i-a_i = |D^i(u)\!\!\downarrow|$. But each letter in
$B^i(u)\!\!\uparrow C^i(u)\!\!\downarrow D_1^i(u)$ will be greater than or equal 
to its corresponding letter in $D^i(u)\!\!\downarrow$, 
so that such letters will contribute 
\begin{eqnarray*}
&&\Sigma(B^i(u)\!\!\uparrow)+\Sigma( C^i(u)\!\!\downarrow) + \Sigma( D_1^i(u)) 
-( \Sigma( D_1^i(u))+\Sigma( D_2^i(u))) = \\
&&\Sigma(B^i(u)\!\!\uparrow)+\Sigma( C^i(u)\!\!\downarrow)- \Sigma( D_2^i(u))
\end{eqnarray*}
to $d_i(v)$. However 
the letters in $A^i(u)\!\!\uparrow$ will be compared to letters 
that lie in either $C^i(u)\!\!\downarrow$, $B^i(u)\!\!\uparrow$, or later 
letters in $A^i(u)\!\!\uparrow$ and hence they will contribute 
0 to $d_i(v)$. Thus 
\begin{eqnarray*}
d_i(v)+\Sigma(s_i(v))&=&  \Sigma(B^i(u)\!\!\uparrow)+\Sigma( C^i(u)\!\!\downarrow)- \Sigma( D_2^i(u))+  \Sigma( D_2^i(u))\\
&=&   \Sigma(B^i(u)\!\!\uparrow) + \Sigma(C^i(u)\!\!\downarrow) =\Sigma(s_i(u)).
\end{eqnarray*}
\end{proof}

We are now ready for the result referred to immediately before 
Lemma \ref{lem:rearrangements}. 

\begin{theorem}\label{thm:rearrangements}
If $u,v \in \P^*$ have increasing/decreasing factorizations, 
then $u \backsim v$ if and only if $u$ is a rearrangement of $v$.
\end{theorem}
\begin{proof}
Suppose $u,v \in \P^*$ have increasing/decreasing factorizations. 
If $u$ is a rearrangement of $v$, then 
$S(u;t,x) = S(v;t,x)$ by Theorem \ref{thm:S} and Lemma 
\ref{lem:rearrangements}. Hence $u \backsim v$.

For the converse, suppose $u \backsim v$.  Since we've just shown that  
a word with an increasing/decreasing factorization is Wilf equivalent to 
any rearrangement of itself with an increasing/decreasing factorization, it 
suffices to consider the case when $u$ and $v$ are both nondecreasing, and to 
show that $u = v$.  
So let $u = u_1 u_2 \cdots u_n$ and $v = v_1 v_2 \cdots v_n$ be nondecreasing.
First note that $u \backsim v$ implies $\text{wt}(u) = \text{wt}(v)$ since for 
any word $w$, the
minimum powers of $x$ and $t$ in $F(w;t,x)$ are $\Sigma(u)$ and $|u|$, 
respectively.  
So the numerators of the expressions for $S(u;t,x) = S(v; t,x)$ in Theorem~\ref{thm:S}
are equal.  
Equating the denominators,  and noting that 
 $d_i = 0$ for all 
$i$ for both $u$ and $v$, we have 
\begin{eqnarray*}
&&t^n x^{\Sigma(v)} + (1-x-tx) \sum_{i = 1}^n t^{n-i} x^{\sum_{j = i+1}^n v_j }(1-x)^{i-1}\\
&& =t^n x^{\Sigma(u)} + (1-x-tx) \sum_{i = 1}^n t^{n-i} x^{\sum_{j = i+1}^n u_j }(1-x)^{i-1}.\end{eqnarray*}
Simplifying, this becomes
$$
 \sum_{i = 1}^n t^{n-i} x^{\sum_{j = i+1}^n v_j }(1-x)^{i-1} =  \sum_{i = 1}^n t^{n-i} x^{\sum_{j = i+1}^n u_j }(1-x)^{i-1}.
$$
  Hence for each $i$, $1 \leq i \leq n$, we have
$$
 x^{\sum_{j = i+1}^m v_j } = x^{\sum_{j = i+1}^n u_j }, 
$$
and therefore $u = v$.   
\end{proof}

Since the values of  $d_i + \Sigma(s_i)$ determine equivalence for
those words with increasing/decreasing factorizations, it is natural 
to ask the same question about those that do not.  Unfortunately, 
equality 
of $d_i + \Sigma(s_i)$ for all $i$ is not enough to determine equivalence in general.  
For example, it was shown in \cite{KLRS} 
that 24153 and 24315 are not Wilf equivalent, but both have the following values:
$$
\begin{array}{c | c}
i & d_i + \Sigma(s_i) \\
\hline
1 & 13 \\
2 & 10 \\
3 & 9 \\
4 & 8
\end{array}.
$$
However, we have not found two words that are Wilf equivalent that have different values of  $d_i + \Sigma(s_i)$.   In particular, the  equivalences proved by Kitaev, Liese, Remmel, and Sagan \cite{KLRS}
 all preserve equality between the  $d_i + \Sigma(s_i)$'s.

\section{Connections with some known sequences} \label{sec:sequences}
In this section we calculate the generating functions $S(u;t,x)$, $A(u;t,x)$, and $F(u;t,x)$ for
some simple words.  Many of the coefficients for these generating functions appear in the OEIS and we will provide the sequence number in various situations.
  For example, suppose that $u$ is a word consisting of a single digit $i\geq1$.  
  Then from  Theorem~\ref{thm:S} we obtain that
$$
S(i;t,x)=\frac{tx^i}{tx^i+(1-x-tx)}
$$
and
$$
A(i;t,x) = \frac{1-x}{1-x-tx}(1-S(i;t,x)) = \frac{1}{1-t\sum_{j=1}^{i-1}x^{j}}.
$$
If we let $t=1$ and $i=3, 4\ldots$, we obtain some familiar generating functions:
\begin{eqnarray*}
A(3;1,x) &=& \frac{1}{1-x-x^2}= 1+x+2x^2+3x^3+5x^4+8x^5+13x^6+\cdots\\
A(4;1,x) &=& \frac{1}{1-x-x^2-x^3}= 1+x+2x^2+4x^3+7x^4+13x^5+24x^6+\cdots\\
A(5;1,x) &=& \frac{1}{1-x-x^2-x^3-x^4}= 1+x+2x^2+4x^3+8x^4+15x^5+29x^6+\cdots\\
A(6;1,x) &=& \frac{1}{1-x-x^2-x^3-x^4-x^5}=
1+x+2x^2+4x^3+8x^4+16x^5+31x^6+\cdots .\\
\end{eqnarray*}
The coefficients of $A(3;1,x)$, $A(4;1,x)$, $A(5;1,x)$, $A(6;1,x)$ are the Fibonacci numbers (sequence \seqnum{A000045} in OEIS), Tribonacci numbers (\seqnum{A000073}) , Tetranacci numbers (\seqnum{A000078}), and Pentanacci numbers (\seqnum{A001591}), respectively.  In general, the coefficient of $x^j$ in $A(i;1,x)$ is $F^{i-1}_{j+1}$,the $(i-1)$-step Fibonnaci number.  The $n$-step Fibonacci number is defined by $F^n_k=0$ for $k\leq0$, $F^n_1=F^n_2=1$, and all other terms by the recurrence $$F^{n}_k=\sum_{i=1}^nF^n_{k-i}.$$  This fact is easily verified by classifying words in $\A(i)$ by their last digit.  If we expand $S(i;1,x)$ as a series we obtain
\begin{eqnarray*}
S(3;1,x) &=& \frac{x^3}{1-2x+x^3}= x^3(1+2x+4x^2+7x^3+12x^4+20x^5+33x^6+\cdots)\\
S(4;1,x) &=& \frac{x^4}{1-2x+x^4}= x^4(1+2x+4x^2+8x^3+15x^4+28x^5+52x^6+\cdots)\\
S(5;1,x) &=& \frac{x^5}{1-2x+x^5}= x^5(1+2x+4x^2+8x^3+16x^4+31x^5+60x^6+\cdots)\\
S(6;1,x) &=& \frac{x^6}{1-2x+x^6}=x^6(
1+2x+4x^2+8x^3+16x^4+32x^5+63x^6+\cdots).\\
\end{eqnarray*}
The coefficients of $S(3;1,x)$, $S(4;1,x)$, $S(5;1,x)$, and $S(6;1,x)$ are partial sums of Fibonacci (\seqnum{A000071}), Tribonacci (\seqnum{A008937}), Tetranacci (\seqnum{A107066}), and Pentanacci (\seqnum{A001949}) numbers, respectively.  In fact, the coefficients of $S(i;1,x)$ are the partial sums of the $(i-1)$-step Fibonacci numbers and can be found in $i$-th column of the array defined in \seqnum{A172119}.  It is also easy to verify this fact by classifying words in $\S(i)$ by their last digit.  Lastly, if we expand $F(i;1,x)$ as a series we obtain
\begin{eqnarray*}
F(3;1,x) &=& \frac{x^3}{1-3x+x^2+2x^3}= x^3(1+3x+8x^2+19x^3+43x^4+94x^5+201x^6+\cdots)\\
F(4;1,x) &=& \frac{x^4}{1-3x+x^2+x^3+2x^4}= x^4(1+3x+8x^2+20x^3+47x^4+107x^5+\\
&&~~~~~~~~~~~~~~~~~~~~~~~~~~~~~~~~~~~~~~~238x^6+\cdots)\\
F(5;1,x) &=& \frac{x^5}{1-3x+x^2+x^3+x^4+2x^5}= x^5(1+3x+8x^2+20x^3+48x^4+\\
&&~~~~~~~~~~~~~~~~~~~~~~~~~~~~~~~~~~~~~~~~~~~~~~111x^5+251x^6+\cdots)\\
F(6;1,x) &=& \frac{x^6}{1-3x+x^2+x^3+x^4+x^5+2x^6}=
x^6(1+3x+8x^2+20x^3+48x^4+\\
&&~~~~~~~~~~~~~~~~~~~~~~~~~~~~~~~~~~~~~~~~~~~~~~~~~~112x^5+255x^6+\cdots).\\
\end{eqnarray*}
The coefficient of $x^j$ for $F(3;1,x)$, $F(4;1,x)$, $F(5;1,x)$, and $F(6;1,x)$ is $2^{j-1}$ minus $F^2_{j+1}$ (\seqnum{A008466}), $2^{j-1}$ minus $F^3_{j+1}$ (\seqnum{A050231}), $2^{j-1}$ minus $F^4_{j+1}$ (\seqnum{A050232}) and $2^{j-1}$ minus $F^5_{j+1}$ (\seqnum{A050233}) respectively.  In general, the coefficient of $x^j$ in $F(i;1,x)$ is $2^{j-1}$ minus $F^{i-1}_{j+1}$.  This fact is also easily verified as the total number of words of weight $j$ is simply $2^{j-1}$ (the number of compositions of $j$) and we can subtract off the words that avoid $u$ which has already been shown to be an $(i-1)$-step Fibonacci number to obtain the number of words that embed $u$.

We now turn to a different class of words.
Suppose that $u = r~r+s~r$ where $r,s \geq 1$.  Then in
the notation of Theorem 1, $s_1 = r+s~r$, $s_2 =r$ and $s_3 = \epsilon$.
To compute $d_1(u)$ and $d_2(u)$, consider the arrays
$$
\begin{array}{cccccccccccc}
r&r+s&r&   & & & &  r&r+s&r& \\
 &   &r&r+s&r& & &   &r&r+s&r.
\end{array}
$$

It is easy to see from these arrays that
$d_1(u) =0$ and $d_s(u) = s$. Thus $d_1(u)+\Sigma(s_1) = 2r+s$ and
$d_2(u)+\Sigma(s_2) = r+s$. By definition $d_3(u) = 0$ so that
$d_3(u)+\Sigma(s_3) = 0$.
Thus by Theorem~\ref{thm:S}
$$
S(r~r+s~r;t,x) = \frac{t^3x^{3r+s}}{t^3x^{3r+s} + (1-x-xt)(t^2x^{2r+s} +
tx^{r+s}(1-x) + (1-x)^2)}
$$
and
$$
A(r~r+s~r;t,x) = \frac{1-x}{1-x-xt}(1-S(r~r+s~r;t,x)).
$$
Now when $t=1$ and $r=1$, these simplify to
\begin{eqnarray*}
S(1~1+s~1;1,x) &=&  \frac{x^{3 +s}}{(1-x)^3(1- \sum_{i=1}^s x^i)} \\
A(1~1+s~1;1,x) &=& \frac{1-2x +x^2 +x^{s+1}}{(1-x)^3(1- \sum_{i=1}^s x^i)}.
\end{eqnarray*}
We can expand these functions as power series around $x=0$ and find that
\begin{eqnarray*}
S(121;1,x) &=& x^4+4 x^5+10 x^6+20 x^7+35 x^8+56 x^9+84 x^{10}+120 x^{11}+165 x^{12}+\\
&&220 x^{13}+286 x^{14}+364 x^{15}+455 x^{16}+560 x^{17}+680
x^{18}+\cdots,\\
S(131;1,x) &=& x^5+4 x^6+11 x^7+25 x^8+51 x^9+97 x^{10}+
176 x^{11}+309 x^{12}+ 530 x^{13}+\\
&&894 x^{14}+1490 x^{15}+2462 x^{16}+4043
x^{17}+6610 x^{18}+\cdots,\\
S(141;1,x) &=& x^6+4 x^7+11 x^8+26 x^9+56 x^{10}+114 x^{11}+224 x^{12}+430 x^{13}+813 x^{14}+\\
&&1522 x^{15}+2831
x^{16}+5244 x^{17}+9688 x^{18}+\cdots,
\mbox{and} \\
S(151;1,x) &=& x^7+4 x^8+11 x^9+26 x^{10}+57 x^{11}+119 x^{12}+241 x^{13}+479 x^{14}+941 x^{15}+\\
&&1835x^{16}+3562 x^{17}+6895 x^{18}+\cdots.
\end{eqnarray*}
Now, the sequence $1,4,10,20,35,56, \ldots$ of coefficients starting at $x^4$ for
$S(121;1,x)$ 
is the sequence of Tetrahedral numbers, defined by $a(n) = \frac{n(n+1)(n+2)}{6}$
(\seqnum{A000292}). Thus we obtain a new combinatorial
interpretation of these numbers. That is, $a(n)$ equals
the number of words $u$ such that $\sum (u) = 3+n$ and
$u \in \mathcal{S}(121)$.
Similarly the sequence $1,4,11,25,51,97, \ldots$ of coefficients starting at $x^5$ for
$S(131;1,x)$  appears in the OEIS as 
sequence \seqnum{A014162}. In this case, with an off-set of 4,
these numbers $b(n)$ count the number of 132-avoiding two-stack stortable
permutations which contain exactly one subsequence of type
51234. See Egge and Mansour \cite{EM}.
Again, we obtain a new combinatorial
interpretation of these numbers. That is, $b(n)$ equals
the number of words $u$ such that $\sum (u) = 4+n$ and
$u \in \mathcal{S}(131)$. The sequences of
coefficients for $S(141,1,x)$ and $S(151,1,x)$ have not
previously appeared in the OEIS.

Similarly, one can expand $A(1~s+1~1;1,x)$ as a power series
about $x=0$ and
find that
\begin{eqnarray*}
A(121;1,x) &=& 1+x+2 x^2+4 x^3+7 x^4+11 x^5+16 x^6+22 x^7+29 x^8+37 x^9+46 x^{10}+\\
&&56 x^{11}+67 x^{12}+79 x^{13}+92 x^{14}+106 x^{15}+\cdots,\\
A(131;1,x) &=& 1+x+2 x^2+4 x^3+8 x^4+15 x^5+27 x^6+47 x^7+80 x^8+134 x^9+222 x^{10}+\\
&&365 x^{11}+597 x^{12}+973
x^{13}+1582 x^{14}+2568 x^{15}+\cdots,\\
A(141;1,x) &=& 1+x+2 x^2+4 x^3+8 x^4+16 x^5+31 x^6+59 x^7+111
x^8+207 x^9+384 x^{10}+\\
&&710 x^{11}+1310 x^{12}+ 2414 x^{13}+4445 x^{14}+8181 x^{15}+\cdots,\ \mbox{and}\\
A(151;1,x) &=& 1+x+2
x^2+4 x^3+8 x^4+16 x^5+32 x^6+63 x^7+123 x^8+239 x^9+463 x^{10}+\\
&&895 x^{11}+1728 x^{12}+3334 x^{13}+6430 x^{14}+12398 x^{15}+\cdots.
\end{eqnarray*}
The coefficient of $x^n$ in
$A(121;1,x)$ is $a(n) = \frac{n(n-1)}{2}+1$.  These numbers are 
the central polygonal numbers (\seqnum{A000124}) and $a(n+1)$ is the maximal number of pieces
obtained when slicing a pancake with $n$ cuts. Thus we obtain
a new combinatorial interpretation of these numbers as
the number of words of $u$ such that $\sum (u) =n$ and $u$ does
not embed 121. The sequence of coefficients in the expansion
of $A(131;1,x)$ starting at $x$ is $1,2,4,8,15,27,47, \ldots$
(\seqnum{A000126}) and counts the number of
ternary numbers with no 0 digit and at least one 2 digit.
Thus we obtain
a new combinatorial interpretation of these numbers as
the number of words of $u$ such that $\sum (u) =n$ and $u$ does
not embed 131. The sequence of coefficients in the expansion
of $A(141;1,x)$ starting at $x$ is $1,2,4,8,16,31,59, \ldots$
(\seqnum{A007800}) is said to have come from a problem
in AI planning and satisfies a recurrence
$a(n) = 4+a(n-1)+a(n-2)+a(n-3)+a(n-4)-a(n-5) -a(n-6) -a(n-7)$ for
$n > 7$.  Thus we obtain
a new combinatorial interpretation of these numbers as
the number of words of $u$ such that $\sum (u) =n$ and $u$ does
not embed 141. The sequence of coefficients in the expansion
of $A(151;1,x)$ starting at $x$ is $1,2,4,8,16,32,63,123, \ldots$
(\seqnum{A145112}) and counts the number of binary word with
fewer that four 0 digits between any pair of consecutive 1 digits.
Thus we obtain
a new combinatorial interpretation of these numbers as
the number of words of $u$ such that $\sum (u) =n$ and $u$ does
not embed 151. One can also check that the
sequence of coefficients in the expansion
of $A(161;1,x)$ starting at $x$ is sequence \seqnum{A145113} and counts the number of binary word with
fewer that five 0 digits between any pair of consecutive 1 digits.

Another simple example is $S(123;t,x)$. In this case
it is easy to see that $d_1(123) = d_2(123) = d_3(123) =0$.
Thus it follows that
\begin{eqnarray*}
S(123;t,x) &=&
\frac{t^3x^6}{t^3x^6+(1-x-xt)(t^2x^5+tx^3(1-x)+(1-x)^2)} \ \mbox{and} \\
A(123;t,x) &=& \frac{1-x}{1-x-xt}(1 -S(123;t,x)).
\end{eqnarray*}
One can compute that
\begin{eqnarray*}
S(123;t,x) &=&
\frac{x^6}{(1-x)^2(x^4-x^3+2x-1)} \ \mbox{and} \\
A(123;t,x) &=& \frac{1-2x+x^2+x^3-x^4+x^5}{(1-x)^2(x^4-x^3+2x-1)}.
\end{eqnarray*}
Expanding these functions as power series  about $x=0$ and letting $t=1$, we  obtain
that
\begin{eqnarray*}
S(123;1,x) &=& x^6+4x^7+11x^8+25x^9+52x^{10}+103x^{11} +199x^{12}+\cdots
 \mbox{and} \\
A(123;1,x) &=& 1+x+2 x^2+4 x^3+8 x^4+16 x^5+31 x^6+59 x^7+111 x^8+\\
&&208 x^9+389 x^{10}+727 x^{11}+1358 x^{12}+\cdots.
\end{eqnarray*}
In this case, neither sequence of coefficients have appeared in
the OEIS.

\section{Wilf equivalence for words of length 3} \label{sec:3}

We now turn to the classification of the Wilf equivalence classes for all words of length
3 in $\mathcal{P}_1 = (\mathbb{P}, \leq)$.    Theorem~\ref{thm:S} will provide the necessary information 
for words with increasing/decreasing factorizations. 
The only words of length 3 without increasing/decreasing 
factorizations are words of the form $bac$ or $cab$ where $a < b \leq c$.
But Kitaev, Liese, Remmel, and Sagan (Lemma 4.1 of \cite{KLRS}) show that 
any word is Wilf equivalent to its reverse, so it suffices to consider $bac$.  To that end, we 
 give an explicit formula 
for $S(bac;t,x)$ where  $a < b \leq c$.

\begin{theorem}\label{thm:bac}  For positive integers $a < b \leq c$, 
$$
S(bac;t,x) = \frac{t^3 x^{a+b+c}  (1 + t x^c(1 + x + \cdots + x^{b-a-1}))}
{(1-x-tx)\psi_{a,b,c}(t,x) + t^3 x^{a+b+c}  (1 + t x^c(1 + x + \cdots + x^{b-a-1})) }
$$
where 
$$
\psi_{a,b,c}(t,x) = (1-x)^2 + t x^c(1-x) + t^2 x^{a+c} + t^3 x^{a+2c}(1 + x + \cdots +
x^{b-a-1}).
$$
\end{theorem}
\begin{proof}   We start with the following expression, which follows from (\ref{eq:S2}) and Lemma~\ref{lem:S}, 
with extra terms to account for the fact that $bac$ does not have an increasing/decreasing 
factorization:
\begin{eqnarray*}
S(bac;t,x) & = & A(bac;t,x) \: t^3 \frac{x^{a+b+c}}{(1-x)^3} - S(bac;t,x)\: t^2 \frac{x^{a+c}}{(1-x)^2} 
-S(bac;t,x)\: t \frac{x^c}{1-x}\\
&& + A(bac;t,x) \: t^4 \frac{x^{b+2c}}{(1-x)^3} (x^a + \cdots + x^{b-1})\\
&&- S(bac;t,x) \: t^3 \frac{x^{2c}}{(1-x)^2} (x^a + \cdots + x^{b-1}).\\
\end{eqnarray*}
The first term, 
$$
A(bac;t,x) \: t^3 \frac{x^{a+b+c}}{(1-x)^3},
$$
is the generating function for words consisting of an embedding of $bac$ appended to 
a word that avoids $bac$.  This includes the words in $\S(bac)$, but also includes 
words that end in overlapping embeddings of $bac$, either in the last four or five characters 
(and do not embed $bac$ prior to those embeddings).  So we need to remove
the terms associated with these words. 
First consider words $w$ that end in overlapping embeddings in the last five characters, as shown:
$$
\begin{array}{rlllllll}
w= & \cdots &  b^+  & a^+  & c^+ &  a^+  &  c^+ \\
\hline
 && b  & a  & c &    & \\
      &  &   &     &  b & a & c 
\end{array},
$$
where for a positive integer $m$, $m^+$ represents any integer greater than or equal
to $m$. 
Since $b \leq c$, these words can be formed by appending an embedding of $ac$ to words in $\S(bac)$.  
So the second term on the right hand side, 
$$
S(bac;t,x)\: t^2 \frac{x^{a+c}}{(1-x)^2},
$$
accounts for these words. 

Now consider words that end in overlapping embeddings in the final four characters only:
$$
\begin{array}{rlllllll}
w= & \cdots  & b^+  & b^+ &  c^+  &  c^+ \\
\hline
 && b  & a  & c &    & \\
        &   &     &  b & a & c 
\end{array}.
$$
The third term, 
$$
S(bac;t,x)\: t \frac{x^c}{1-x},
$$
removes the terms associated with these words by appending an embedding of $c$ to words
in $\S(bac)$.  However, because $a < b$, this 
 also includes terms associated with 
words of the following form:
$$
\begin{array}{rlllllll}
w= & \cdots  & b^+  & [a,b) &  c^+  &  c^+ \\
\hline
 && b  & a  & c &    & \\
        &   &     &  b & a & c 
\end{array},
$$
that is, words that first embed $bac$ starting at the fourth character
from the end, and end in an embedding of $bacc$ but not $bbcc$.  
To correct for these terms, consider the fourth term on the right hand side:
$$
A(bac;t,x) \: t^4 \frac{x^{b+2c}}{(1-x)^3} (x^a + \cdots + x^{b-1}).
$$
This is the generating function for those words that end in an embedding
of $bacc$ but not $bbcc$ (hence the $x^a + \cdots + x^{b-1}$ term), 
appended to words that avoid $bac$.  So this includes the 
words that we want, but also  may include words that first 
embed $bac$ beginning at the fifth or sixth character from the end.
However, the first of these situations is impossible, as shown,
$$
\begin{array}{rlllllll}
w= & \cdots & \bullet & b^+  & [a,b) &  c^+  &  c^+ \\
\hline
&& b & a & c \\
 &&& b  & a  & c &    & \\
      &  &   &     &  b & a & c 
\end{array},
$$
since a character in $[a,b)$ cannot embed $c$.   The final term, 
$$
S(bac;t,x) \: t^3 \frac{x^{2c}}{(1-x)^2} (x^a + \cdots + x^{b-1}),
$$
accounts for the second possibility,
$$
\begin{array}{rlllllllll}
w= & \cdots & b^+ & a^+ & c^+  & [a,b) &  c^+  &  c^+ \\
\hline
&& b & a & c \\
 &&&& b  & a  & c &    & \\
      &  &&   &     &  b & a & c 
\end{array},
$$
by appending an embedding of $acc$, whose first character is in $[a,b)$, to 
words in $\S(bac)$.  

We can now solve for $S(bac;t,x)$:
\begin{eqnarray*}
 S(bac;t,x) & =&  \frac{t^3 \frac{x^{a+b+c}}{(1-x)^3} + t^4 \frac{x^{b+2c}}{(1-x)^3} (x^a + \cdots + x^{b-1})}{1+ t \frac{x^c}{1-x} + t^2 \frac{x^{a+c}}{(1-x)^2} + t^3 \frac{x^{2c}}{(1-x)^2} (x^a + \cdots + x^{b-1})} A(bac;t,x) \\ \\
& = & \frac{t^3 x^{a+b+c} + t^4 x^{b+2c}(x^a + \cdots + x^{b-1})}
{(1-x)^3 + t x^c (1-x)^2 + t^2 x^{a+c} (1-x) + t^3 x^{2c} (1-x) (x^a + \cdots + x^{b-1})} \\\\
&& \cdot  A(bac;t,x).
\end{eqnarray*}
Substituting $A(bac;t,x)=\frac{1-x}{1-x-tx} (1-S(bac;t,x))$, we obtain
\begin{eqnarray*}
S(bac;t,x) & =&\frac{t^3 x^{a+b+c} + t^4 x^{b+2c}(x^a + \cdots + x^{b-1})}
{(1-x)^2 + t x^c (1-x) + t^2 x^{a+c}+ t^3 x^{2c} (x^a + \cdots + x^{b-1})}\\
&&  \cdot \frac{1}{1-x-tx} (1-S(bac;t,x)).
\end{eqnarray*}
Solving for $S(bac;t,x)$ and factoring appropriate terms gives the result. 
\end{proof} 

As a corollary, we can now classify Wilf equivalence for words of the form $bac$ with $a < b \leq c$.


\begin{corollary}\label{cor:bac}
For positive integers $a < b \leq c$, the only words  
 Wilf equivalent to $bac$ are $bac$ and
$cab$.  
\end{corollary}

\begin{proof}   As noted previously, $bac \backsim cab $ since they are reverses of each other.  So it remains
to show that no other words of length three are Wilf equivalent to these two.
First, note that $1 + t x^c (1+x + \cdots + x^{b-a-1})$ does not 
divide the denominator of the expression for $S(bac;t,x)$ in 
Theorem~\ref{thm:bac}
since, for example, it doesn't divide it when $x=1$.  So $S(bac;t,x)$ does not have 
a single monomial in the numerator, and therefore is not equal to $S(u;t,x)$ for any $u$ that
 has an increasing/decreasing factorization by Theorem~\ref{thm:S}.  So 
 suppose $bac \backsim b'a'c'$ with $a' < b' \leq c'$.   We'll show that $a = a'$, $b=b'$
 and $c = c'$.

  Equating $S(bac;t,x)$ and 
 $S(b'a'c';t,x)$ from Theorem~\ref{thm:bac}, we have
 \begin{eqnarray*}
&&\phi_{a,b,c}(t,x) \left[ (1-x-tx)\psi_{(a',b',c')}(t,x) + \phi_{a',b',c'}(t,x)\right]\\
&&  = \phi_{a',b',c'}(t,x) \left[ (1-x-tx)\psi_{(a,b,c)}(t,x) + \phi_{a,b,c}(t,x)\right],
  \end{eqnarray*}
  where $\phi_{a,b,c}(t,x) =  t^3 x^{a+b+c}(1+tx^c(1+x+ \cdots + x^{b-a-1}))$, 
 and similarly for $\phi_{a',b',c'}(t,x)$.
Since $bac \backsim b'a'c'$, we have $a+b+c = a' + b' + c'$, so we may simplify to
 \begin{eqnarray*}
&& (1+tx^c(1+x+ \cdots + x^{b-a-1}))\psi_{(a',b',c')}(t,x) \\ 
&& = (1+tx^{c'}(1+x+ \cdots + x^{b'-a'-1}))\psi_{(a,b,c)}(t,x).
 \end{eqnarray*}
Recalling that 
 $$ \psi_{(a,b,c)}(t,x) = (1-x)^2 + t x^c(1-x) + t^2 x^{a+c} + t^3 x^{2c}(x^a + \cdots +
x^{b-1}), $$
and equating powers of $t^2$ on both sides, we have
$$
x^{a'+c'}+x^{c+c'}(1-x)(1+x+ \cdots + x^{b-a-1}) = 
x^{a+c}+x^{c+c'}(1-x)(1+x+ \cdots + x^{b'-a'-1}) 
$$
Since $a<c$, the smallest power of $x$ on the left is $a'+c'$, and the smallest on the right is
$a+c$.  So 
$$
a'+c'= a+c.
$$
Since we know $a+b+c = a' + b' + c'$, this gives $b = b'$. 
Now equating the largest powers of $x$, we have
$$
c+c' +b-a = c+c'+b'-a',
$$
which gives $a = a'$, and therefore $c = c'$. 
\end{proof}

We're now ready to completely classify Wilf equivalence of words of length 3.

\begin{theorem} \label{thm:length3}  Wilf equivalence relative to $\mathcal{P}_1$ partitions
$\P^3$ into the following equivalence classes.  

\begin{enumerate}

\item $\{aaa\}$
 for any $a \in \P$  

\item $\{aab, aba, baa\}$  if $a <  b$

\item $\{aab, baa\}$ and $\{aba\}$  if $a > b$

\item $\{bac, cab\}$  and $\{abc, acb, cba, bca\}$ if $a < b < c$.

\end{enumerate}

\end{theorem}
\begin{proof}  Theorem~\ref{thm:rearrangements} establishes the equivalence among words in the sets $\{aab,aba,baa\}$  in case 2, the sets $\{aab, baa\}$ in 
case 3, and $\{abc, acb, cba, bca\}$ in case 4, as well as the fact that these sets are all in distince Wilf equivalences classes.
The remaining cases,   $\{bac, cab\}$ in case 4 and $\{aba\}$ in case 3,  follow from Corollary~\ref{cor:bac}. 
\end{proof}

Note that when we consider the permutations of $S_3$, there are only 
two Wilf equivalence classes, namely, $\{123,132,321,231\}$ and 
$\{213,312\}$. We computed $S(123;t,x)$ and $A(123;t,x)$
in the previous section. Thus to complete the possibilities 
for $S(\sg;t,x)$ for $\sg \in S_3$, we need only compute 
$S(213;t,x)$ and $A(213;t,x)$. In this case, we must use 
Theorem \ref{thm:bac} from which we obtain that  
\begin{eqnarray*}
S(213;t,x)&=& \frac{t^3x^6(1+tx^3)}{(1-x-xt)((1-x)^2 +tx^3(1-x)+t^2x^4+t^3x^7)+t^3x^6(1+tx^3)} \ \mbox{and} \\
A(213;t,x)&=& \frac{1-x}{1-x-xt}(1-S(213;t,x)).
\end{eqnarray*}
One can compute that 
\begin{eqnarray*}
S(213;1,x)&=& \frac{x^6+x^8}{1-4x+5x^2-x^3-2x^4+2x^6+x^7-2x^8} \ \mbox{and}\\
A(213;1,x)&=& \frac{(1-x)(1-4x+5x^2-x^3-2x^4+x^6+x^7-3x^8)}
{(1-2x)(1-4x+5x^2-x^3-2x^4+2x^6+x^7-2x^8)}.
\end{eqnarray*}
In this case, if one expands these functions as power series 
 about $x=0$, one obtains 
\begin{eqnarray*}
S(213;1,x)&=& x^6+4 x^7+11 x^8+26 x^9+55 x^{10}+109 x^{11}+207 x^{12}+381 x^{13}+\\
&&684 x^{14}+1201 x^{15}+O[x]^{16} \ \mbox{and}\\
A(213;1,x)&=& 1+x+2 x^2+4 x^3+8 x^4+16 x^5+31 x^6+59 x^7+111 x^8+207 x^9+\\
&&385 x^{10}+716 x^{11}+1334 x^{12}+2494 x^{13}+4685 x^{14}+8853 x^{15}
+O[x]^{16}.
\end{eqnarray*}
However, neither of the two sequence of coefficients have appeared 
in the OEIS.

\section{The strong rearrangement conjecture}  \label{sec:rearrangement}

In this section we discuss the strong rearrangement 
conjecture and its connection to the  family 
of finite posets
$\mathcal{P}_{[m]} = ([m]^*, \leq)$.  We  also give an analogue of Theorem
\ref{thm:S} for  $S(u;x_1, \ldots ,x_n)$. 
Our first result relates Wilf equivalence in $[m]^*$ to Wilf equivalences in $\mathbb{P}^*$ that are witnessed by rearrangement maps.

\begin{theorem}\label{thm:lift}  Suppose $u,v \in [m]^*$ for some positive integer
$m$.  Then $u \backsim_{[m]} v$ if and only if there exists a rearrangement map $f: \mathbb{P}^* \rightarrow \mathbb{P}^*$ that witnesses the Wilf equivalence $u \backsim v$.  
\end{theorem}

\begin{proof}  First note that if there is a rearrangement map $f: \mathbb{P}^* \rightarrow \mathbb{P}^*$ that witnesses the Wilf equivalence $u \backsim v$, then the restriction of $f$ to $[m]^*$ is a $W_{[m]}$-preserving bijection that shows $u \backsim_{[m]} v$.

For the converse, suppose $u, v \in [m]^*$ and $u \backsim_{[m]} v$, so that 
$F(u; x_1, \ldots , x_m) = F(v; x_1, \ldots , x_m)$.  Then 
 there is a $W_{[m]}$-preserving 
bijection $g:\F(u) \cap [m]^* \rightarrow \F(v) \cap [m]^*$.
So $g(w)$ is a rearrangement of $w$ for all $w$.  
This bijection 
can then be lifted to the desired rearrangement $f$, as follows.  Suppose
$w = w_1 \cdots w_n \in \P^*$ and $1 \leq i_1 < \cdots < i_l \leq n$ is 
the sequences of indices $i$ such that $w_i \geq m$.  Then let 
$\overline{w}$ be the word in $[m]^*$ that results by replacing each 
$w_{i_k}$ with $m$.  Then $u \leq w$ if, and only if, $u \leq \overline{w}$.  
Now apply $g$ to $\overline{w}$.  Then since $z = g(\overline{w})$ is a rearrangement
of $\overline{w}$, there is a sequence $1 \leq j_1 < \cdots < j_l \leq n$ 
consisting of all the indices $j$ such that $z_j = m$.   Then let $f(w)$ be 
the result of replacing $z_{j_k}$ by $w_{i_k}$ for $k = 1, \ldots, l$. 
\end{proof}

Theorem~\ref{thm:lift} shows that the question of whether
$u \backsim v$ implies a rearrangement witnessing the equivalence
can be answered by restricting to a finite alphabet.  We have computed $S(u;x_1,\ldots ,x_5)$ for all permutations in $S_n$ for $n \leq 5$ and indeed, if $u \backsim v$ in this case, then 
$S(u;x_1,\ldots ,x_5) = S(v;x_1,\ldots ,x_5)$. Thus the 
strong rearrangement conjecture holds for these words.

Next we consider an analogue of Theorem~\ref{thm:S} for the more refined 
generating functions $S(u;x_1, \ldots ,x_m)$.  It is still the case 
that   
\begin{equation}\label{eq:Sx2}
 \S(u) \cap [m]^*  = (\A(u) \cap [m]^*)(\mathcal{W}(u)\cap [m]^*)- 
\left(\bigcup_{i=1}^{n-1} (\S^{(i)}(u) \cap [m]^*)\right).
\end{equation}
It is easy to see that 
\begin{eqnarray}\label{eq:Sx3}
\sum_{w \in \A(u)\W(u) \cap [m]^*}W_{[m]}(w) &=& 
A(u;x_1,\ldots,x_m)\prod_{r=1}^n \sum_{s=u_i}^m x_j \nonumber \\
&=& \frac{1}{1-\sum_{i=1}^m x_i}(1-S(u;x_1, \ldots, x_m))\prod_{r=1}^n 
\sum_{s=u_r}^m x_s.
\end{eqnarray}
We also have that 
\begin{equation*}\label{eq:Sx4}
\S^{(i)}(u)\cap [m]^* = \bar{\S}^{(i)}(u)\W(s_i(u)) \cap [m]^*.
\end{equation*}
Thus if 
\begin{eqnarray*}
S^{(i)}(u,x_1, \ldots, x_m) &=& \sum_{w \in \S^{(i)}(u) \cap [m]^*} W_{[m]}(w) \ \mbox{and} \\
\bar{S}^{(i)}(u,x_1, \ldots, x_m) &=& \sum_{w \in \bar{\S}^{(i)}(u) \cap [m]^*} W_{[m]}(w),
\end{eqnarray*}
then we will have 
\begin{equation*}\label{eq:Sx5}
S^{(i)}(u,x_1, \ldots, x_m) = \bar{S}^{(i)}(u,x_1, \ldots, x_m)\prod_{r=i+1}^n 
\sum_{s=u_r}^m x_s.
\end{equation*}
The only step in our proof of Theorem \ref{thm:S} which does not 
have an analogue in this case is the fact that 
$$\bar{S}^{(i)}(u;t,x) = x^{d_i(u)} S(u;t,x).$$
It will no longer be the case that 
$\bar{S}^{(i)}(u;x_1, \ldots,x_m)$ is a multiple of $S(u;x_1, \ldots, x_m)$ if 
$d_i(u) > 0$.  However, if $d_i(u) =0$, then it will be the case 
that $\bar{S}^{(i)}(u) \cap [m]^* = \S^{(i)}(u) \cap [m]^*$ so that 
\begin{equation*}\label{eq:Sx6}
 \bar{S}^{(i)}(u,x_1, \ldots, x_m) =S(u,x_1, \ldots, x_m).
\end{equation*}
Thus if $d_i(u) =0$ for all $i =1, \ldots, n-1$, then 
we will have 
\begin{equation}\label{eq:Sx7}
S^{(i)}(u,x_1, \ldots, x_m) = S(u,x_1, \ldots, x_m)\sum_{r=i+1}^n 
\sum_{s=u_r}^m x_s
\end{equation} 
for all $i$.  However, it is easy to see that $d_i(u) =0$ for all $i =1, \ldots, n-1$ if and only if $u_1 \leq \cdots \leq u_n$. In that case, 
we can see from (\ref{eq:Sx2}),  (\ref{eq:Sx3}), and 
(\ref{eq:Sx7}) that 
\begin{eqnarray*}\label{eq:Sx8}
S(u,x_1, \ldots, x_m) &=& \frac{1}{1-\sum_{i=1}^m x_i}(1-S(u;x_1, \ldots, x_m))\sum_{r=1}^n \sum_{s=u_r}^m x_s \\
&&- \sum_{i=1}^{n-1}  S(u,x_1, \ldots, x_m)\sum_{r=i+1}^n 
\sum_{s=u_r}^m x_s.
\end{eqnarray*}
Solving for $S(u,x_1, \ldots, x_m)$ will then result in the following 
theorem.

\begin{theorem} \label{thm:Sfinite}
Suppose $u = u_1 \ldots u_n \in [m]^*$ is weakly increasing.  Then 
\begin{eqnarray*}
S(u;x_1, \ldots, x_m) & = & \frac{\prod_{i=1}^n\sum_{j=u_i}^m x_j}
{\left(1 +  \sum_{i=1}^{n-1}  \prod_{j=i+1}^n\sum_{l=u_j}^m x_l\right) 
( 1 -\sum_{i=1}^m x_i ) + \prod_{i=1}^n\sum_{j=u_i}^m x_j}.
\end{eqnarray*}
\end{theorem}

\section{Further work} \label{sec:conclusion}

Many of the ideas in this paper can be extended to generalized factor order on $\mathbb{P}^*$ with 
other partial orders.  In particular, in \cite{LLR} we consider the mod $k$ partial order on $\mathbb{P}^*$ defined by setting $m \leq_k n$ if $m \leq n$ and $m = n \text{ mod } k$.  For example, the Hasse diagram for the mod 3 partial order consists of the three chains
\begin{eqnarray*}
&&1 \leq_3 4 \leq _3 7 \leq_3 \cdots, \\
&&2 \leq_3 5 \leq _3 8 \leq_3 \cdots, \text{ and}\\
&&3 \leq_3 6 \leq _3 9 \leq_3 \cdots.\
\end{eqnarray*}
An generalization of Theorem~\ref{thm:S} in this context applies to a rich class of words that generalizes the set of words in $\mathbb{P}^*$ with increasing/decreasing factorizations.  One interesting  result of this theorem is that the rearrangement conjectures do not hold in general for the mod $k$ partial order with $k \geq 2$, however we can identify those words for which we believe the rearrangement conjectures do hold.  We refer the reader to \cite{LLR} for details.

\bigskip
\hrule
\bigskip



\noindent Referenced sequences:
\seqnum{A000045},
\seqnum{A000071}
\seqnum{A000073}, 
\seqnum{A000078},
\seqnum{A000124},
\seqnum{A000126},
\seqnum{A000292},
\seqnum{A001591},
\seqnum{A001949},
\seqnum{A007800},
\seqnum{A008466},
\seqnum{A008937},
\seqnum{A014162},
\seqnum{A050231},
\seqnum{A050232},
\seqnum{A050233},
\seqnum{A107066},
\seqnum{A145112},
\seqnum{A145113},
\seqnum{A172119}.

 \end{document}